\documentclass{article}

\usepackage{amsmath, amssymb, amsthm}
\usepackage{hyperref}
\usepackage{cleveref}
\usepackage{chngcntr}
\usepackage{thmtools}
\usepackage{thm-restate}
\usepackage{fancyhdr}
\usepackage{dsfont}
\usepackage[english]{babel}
\usepackage[utf8]{inputenc}
\usepackage{commath}
\usepackage{marvosym}
\usepackage{mathtools}

\usepackage{thmtools}
\usepackage{thm-restate}
\usepackage{soul}

\usepackage{tikz-cd}

\usepackage[margin=1.5in]{geometry}

\usepackage{comment}

\newtheorem{theorem}{Theorem}[section]
\newtheorem{corollary}[theorem]{Corollary}
\newtheorem{proposition}[theorem]{Proposition}

\newtheorem{lemma}[theorem]{Lemma}

\newtheorem*{theorem*}{Theorem}
\newtheorem*{corollary*}{Corollary}
\newtheorem*{lemma*}{Lemma}
\newtheorem*{proposition*}{Proposition}

\theoremstyle{definition}
\newtheorem{definition}[theorem]{Definition}
\newtheorem{remark}[theorem]{Remark}
\newtheorem{example}[theorem]{Example}

\newtheorem{question}[theorem]{Question}

\newtheorem{chunk}[theorem]{}

\DeclareMathOperator{\pdim}{pd}
\DeclareMathOperator{\tor}{Tor}
\DeclareMathOperator{\ext}{Ext}

\DeclareMathOperator{\depth}{depth}
\DeclareMathOperator{\Hom}{Hom}

\DeclareMathOperator{\idim}{id}

\DeclareMathOperator{\gdim}{Gdim}

\DeclareMathOperator{\coker}{coker}

\DeclareMathOperator{\width}{width}
\DeclareMathOperator{\cone}{cone}

\title{The Auslander Bound for Complexes}
\author{Andrew J. Soto Levins}
\date{}

\begin{document}
\maketitle
\begin{abstract}
If a module $M$ has finite projective dimension, then the Ext modules of $M$ against any other module eventually vanish and the projective dimension of $M$ gives a uniform bound for this vanishing. For modules of infinite projective dimension there can still exist a bound. Such a bound is called the Auslander bound. In this paper we define a similar bound for complexes and give applications.
\end{abstract}

\section{Introduction}
In this paper we define and study an Auslander bound for complexes. To motivate the definition of Auslander bound for a module, let $R$ be a ring and let $M$ be an $R$-module. If $\pdim_{R}M<\infty$, then for all modules $N$ we have $\ext_{R}^{n}(M,N)=0$ for $n>\pdim_{R}M$, and so there is a bound on the vanishing of Ext that only depends on $M$. What can be said about the vanishing of Ext when $M$ does not necessarily have finite projective dimension?
\begin{definition} For finitely generated modules $R$-modules $M$ and $N$, set
\begin{enumerate}
\item $P_{R}(M,N)=\sup\{n|\ext_{R}^{n}(M,N)\neq 0\}$
\item $b_{M}^{R}=\sup\{P_{R}(M,N)|N\text{ is finitely generated and }P_{R}(M,N)<\infty\}$.
\end{enumerate}
The number $b_{M}^{R}$ is called the \textit{Auslander bound} of $M$.
\end{definition}
If finite, the number $P_{R}(M,N)$ is the index of the last nonzero Ext of $M$ against $N$. If $\pdim_{R}M<\infty$, then $b_{M}^{R}=\pdim_{R}M$ by \cite[Theorem 2.11.1]{wei}, and so the Auslander bound can be thought of as a generalization of projective dimension.\newline

We say $R$ has AC if $b_{M}^{R}<\infty$ for all finitely generated modules $M$, and $R$ is AB if $R$ is a Gorenstein local ring with AC. Having AC is a condition on the ring and on all modules and many theorems in the literature that use Auslander bounds assume that the ring has AC, but only make conclusions about an individual module or complex. In this paper we prove results by only assuming that a complex has a finite Auslander bound.\newline

In Section \ref{backgroundsection} we review the necessary background for the rest of the paper, and in Section \ref{anauslanderboundforcomplexessection} we prove a proposition that allows us to reduce results about complexes to the module case. In Section \ref{aresultofaraya} we show that the Auslander bound and the Gorenstein dimension of a complex are equal when both are finite, and in Section \ref{thetheoremsofchristensenandholmrevisited} we prove two theorems on the finiteness of homological dimensions for complexes with finite Auslander bound. Finally in Section \ref{deriveddepthformulasection} we prove a new case of the derived depth formula.\newline

I would like to thank Lars Christensen, Henrik Holm, David Jorgensen, Mark Walker, and a referee for their very helpful comments on an earlier draft that greatly improved this paper. I would also like to thank Jack Jeffries, Tom Marley, and Claudia Miller for many helpful conversations throughout this project.

\section{Background} \label{backgroundsection}
In this section $R$ is a Noetherian ring. All complexes are indexed homologically, so if $M$ is a complex, then $M$ has the form
\[M=\dots\rightarrow M_{i+1}\xrightarrow{\partial_{i+1}} M_{i}\xrightarrow{\partial_{i}} M_{i-1}\rightarrow\cdots.\]
The supremum and infimum of $M$ are 
\[\sup{M}=\sup\{i|M_{i}\neq 0\} \text{ and } \inf{M}=\inf\{i|M_{i}\neq 0\}.\] 
The cokernels and kernels of differentials of $M$ are denoted by $C_{i}(M)=\text{coker}{\partial_{i+1}^{M}}$ and $Z_{i}(M)=\text{ker}{\partial_{i}^{M}}$. A shift of $M$ is a complex $\Sigma^{n}M$ with $(\Sigma^{n}M)_{i}=M_{i-n}$ and differential $\partial_{i}^{\Sigma^{n}M}=(-1)^{n}\partial_{i-n}^{M}$. The truncations of $M$ are
\[M_{\leq n}=0\rightarrow M_{n}\rightarrow M_{n-1}\rightarrow\cdots,\quad\quad\quad M_{\geq n}=\cdots\rightarrow M_{n+1}\rightarrow M_{n}\rightarrow 0,\]
and
\[M_{\subset n}=0\rightarrow C_{n}(M)\rightarrow M_{n-1}\rightarrow M_{n-2}\rightarrow\cdots.\]
We denote the derived category of $R$ by $D(R)$. We let $D_{b}(R)$ denote the full subcategory of $D(R)$ consisting of all complexes isomorphic to a bounded complex in $D(R)$ and let $D_{b}^{f}(R)$ denote the subcategory of $D_{b}(R)$ whose complexes have degreewise finitely generated homology. Every homologically bounded complex is isomorphic in $D(R)$ to a bounded complex.\newline

Our main object of study in this paper is the Auslander bound of a complex which uses the Ext modules of two complexes, and semiprojective and semiinjective resolutions are needed to define Ext. These are the analogs of projective and injective resolutions for modules.
\begin{chunk} Let $M$ and $N$ be complexes.
\begin{enumerate}
\item A complex $P$ is \textit{semiprojective} if $P_{i}$ is projective for all $i$ and $\Hom_{R}(P,\_)$ preserves quasiisomorphisms, and $P$ is a \textit{semiprojective resolution} of $M$ if there exists a quasiisomorphism $\pi:P\rightarrow M$. The \textit{projective dimension} of $M$ is defined to be the infimum of $p\in\mathbb{Z}$ for which there exists a semiprojective resolution $P\rightarrow M$ with $P_{i}=0$ for all $i>p$.
\item If $M$ is in $D_{b}^{f}(R)$, then there exists a semiprojective resolution $P\rightarrow M$ with $P_{i}$ finitely generated for all $i$ and $P_{i}=0$ for all $i<\inf{H(M)}$ (see \cite[Theorem 5.2.16]{christensenfoxbyholm}). 
\item A complex $I$ is \textit{semiinjective} if $I_{i}$ is injective for all $i$ and $\Hom_{R}(\_,I)$ preserves quasiisomorphisms, and $I$ is a \textit{semiinjective resolution} of $N$ if there exists a quasiisomorphism $\iota:N\rightarrow I$. The \textit{injective dimension} of $N$ is defined to be the infimum of $-p\in\mathbb{Z}$ for which there exists a semiinjective resolution $N\rightarrow I$ with $I_{p}=0$ for all $i<p$.
\item There exists a semiinjective resolution $N\rightarrow I$ with $I_{i}=0$ for $i>\sup{H(N)}$ (see \cite[Theorem 5.3.26]{christensenfoxbyholm}).
\item Now
\[\ext_{R}^{n}(M,N)=H_{-n}(\mathbb{R}\Hom_{R}(M,N))\text{ and }\tor_{n}^{R}(M,N)=H_{n}(M\otimes_{R}^{\mathbb{L}}N)\]
where 
\[\Hom_{R}(P,N) \simeq \mathbb{R}\Hom_{R}(M,N) \simeq \Hom_{R}(M,I)\]
and
\[P\otimes_{R}N \simeq M\otimes_{R}^{\mathbb{L}}N \simeq M\otimes_{R}Q\]
with $Q\rightarrow N$ a semiprojective resolution.
\item If $R$ is a Noetherian local ring with residue field $k$, then the \textit{depth} of $M$ is
\[\depth{M}=-\sup{H(\mathbb{R}\Hom_{R}(k,M))}\]
and the \textit{width} of $M$ is
\[\width{M}=\inf{H(k\otimes_{R}^{\mathbb{L}}M)}.\]
\end{enumerate}
\end{chunk}

Tools we will use in this paper are Tate homology and Tate cohomology, which are defined for complexes of finite Gorenstein dimension.
\begin{chunk} Let $M$ and $N$ be complexes in $D_{b}^{f}(R)$.
\begin{enumerate}
\item An exact complex $T$ of projective modules is \textit{totally acyclic} if $\Hom_{R}(T,Q)$ is exact for all projective modules $Q$, and a \textit{complete projective resolution} of $M$ is a diagram
\[T\xrightarrow{\tau} P\xrightarrow{\pi} M\]
where $\pi$ is a semiprojective resolution, $T$ is totally acyclic, and $\tau_{i}$ is an isomorphism for $i\gg 0$.
\item The Gorenstein dimension of $M$, denoted by $\gdim_{R}M$, is the least integer $n$ so that there exists a complete projective resolution as above with $\tau_{i}$ an isomorphism for all $i\geq n$. It is well known that $\gdim_{R}M<\infty$ if and only if the canonical map
\[M\rightarrow \mathbb{R}\Hom_{R}(\mathbb{R}\Hom_{R}(M,R),R)\]
is a quasiisomorphism. When $\gdim_{R}M<\infty$, we have
\[\gdim_{R}M=-\inf{H(\mathbb{R}\Hom_{R}(M,R))}\]
(see \cite[Remark 5.4]{veliche}).
\item If $T\rightarrow P\rightarrow M$ is a complete projective resolution of $M$, then the Tate homology of $M$ against $N$ is
\[\widehat{\tor}_{n}^{R}(M,N)=H_{n}(T\otimes_{R}N)\]
and the Tate cohomology of $M$ against $N$ is
\[\widehat{\ext}_{R}^{n}(M,N)=H_{-i}(\Hom_{R}(T,N)).\]
\end{enumerate}
\end{chunk}

\section{The Auslander Bound for Complexes} \label{anauslanderboundforcomplexessection}
In this section we introduce the Auslander bound for complexes and prove Proposition \ref{lemmaforpmrbmrimpliesgdimforcomplexes} which allows us to reduce results about complexes to the module case. Our definition of the following is motivated by \cite[Theorem 8.1.8]{christensenfoxbyholm}.
\begin{definition} For $M$ and $N$ complexes with $\inf{H(N)}$ an integer, define
\[P_{R}(M,N)=\sup\{n\in\mathbb{Z}|\ext_{R}^{n-\inf{H(N)}}(M,N)\neq 0\}=\inf H({N})-\inf H({\mathbb{R}\Hom_{R}(M,N)}).\]
For $M\in D_{b}^{f}(R)$, the \textit{Auslander bound} of $M$ is
\[B_{M}^{R}=\sup\{P_{R}(M,N)|N\in D_{b}^{f}(R)\text{ and }P_{R}(M,N)<\infty\}.\]
\end{definition}

Note that $P_{R}(M,N)$ is invariant under shifting $N$.
\begin{example} Let $R$ be a ring and let $M\in D_{b}^{f}(R)$ be a complex. If $\pdim_{R}M<\infty$, then $B_{M}^{R}=\pdim_{R}M$ by \cite[Theorem 8.1.8]{christensenfoxbyholm}. If $M$ is a module, then $b_{M}^{R}=\pdim_{R}M$ by \cite[Theorem 2.11.1]{wei}, and so $b_{M}^{R}=B_{M}^{R}$.
\end{example}

Before proving Proposition \ref{lemmaforpmrbmrimpliesgdimforcomplexes}, we first need to prove a few lemmas. 
\begin{lemma} \label{prelemmaforpmrbmrimpliesgdimforcomplexes} Let $R$ be a Noetherian ring, let $M\in D_{b}^{f}(R)$ be a complex, and let $P\rightarrow M$ be a semiprojective resolution with $P_{i}$ finitely generated for all $i$. If $n\geq\sup{H(M)}$, then $B_{M}^{R}<\infty$ if and only if $B_{C_{n}(P)}^{R}<\infty$.
\end{lemma}

\begin{proof}
Let $N\in D_{b}^{f}(R)$ be a complex. We can assume $\inf{H(N)}=0$. Now consider the exact sequence
\[\star\quad\quad0\rightarrow P_{\leq n-1}\rightarrow P\rightarrow P_{\geq n}\rightarrow 0\]
when $n\geq\sup{H(M)}$. Since $P_{\geq n}\simeq \Sigma^{n}C_{n}(P)$, the following sequence is exact
\[0\rightarrow \mathbb{R}\Hom_{R}(\Sigma^{n}C_{n}(P),N)\rightarrow \mathbb{R}\Hom_{R}(M,N)\rightarrow \mathbb{R}\Hom_{R}(P_{\leq n-1},N)\rightarrow 0.\]
Since $\pdim_{R}P_{\leq n-1}\leq n-1$, we have $\ext_{R}^{i}(P_{\leq n-1},N)=0$ for $i>n-1$ by \cite[Theorem 8.1.8]{christensenfoxbyholm}, and so 
\[\ext_{R}^{i}(M,N)\cong \ext_{R}^{i}(\Sigma^{n}C_{n}(P),N)\cong \ext_{R}^{i-n}(C_{n}(P),N)\]
for all $i>n$. This isomorphism proves the lemma.
\end{proof}

\begin{lemma} \label{lemmaforbmrbmwhenmismod} Let $R$ be a ring and let $M$ and $N$ be bounded complexes. Assume $P_{R}(M_{i},N_{j})=0$ for all $i$ and all $j$. If $\pi:N\rightarrow I$ is a quasiisomorphism where $I$ is a complex of injectives that is bounded to the left, then 
\[\Hom_{R}(M,\pi):\Hom(M,N)\rightarrow\Hom_{R}(M,I)\]
is a quasiisomorphism. 
\end{lemma}

\begin{proof}
Since $M$ is bounded, by \cite[Proposition 2.6]{christensenfrankildholm} it is enough to show
\[\Hom_{R}(M_{i},\pi):\Hom_{R}(M_{i},N)\rightarrow \Hom_{R}(M_{i},I)\]
is a quasiisomorphism for each $i$. We do this by showing $\Hom_{R}(M_{i},\cone{\pi})$ is exact.\newline

Let $a=\sup{\cone{\pi}}$. Then the complex $\cone{\pi}$ is exact and has the form
\[0\rightarrow (\cone{\pi})_{a}\xrightarrow{\partial_{a}} (\cone{\pi})_{a-1}\xrightarrow{\partial_{a-1}} (\cone{\pi})_{a-2}\rightarrow\cdots\]  
with $(\cone{\pi})_{j}=I_{j}\oplus N_{j-1}$ and $P_{R}(M_{i},(\cone{\pi})_{j})=0$ since $I_{j}$ is injective and $P_{R}(M_{i},N_{j})=0$ for all $i$ and all $j$. Since 
\[0\rightarrow (\cone{\pi})_{a}\rightarrow (\cone{\pi})_{a-1}\rightarrow \text{im}{\partial_{a-1}}\rightarrow 0\]
is exact and since $P_{R}(M_{i},(\cone{\pi})_{a})=P_{R}(M_{i},(\cone{\pi})_{a-1})=0$, we have $P_{R}(M_{i},\text{im}{\partial_{a-1}})=0$, and so
\[0\rightarrow \Hom_{R}(M_{i},(\cone{\pi})_{a})\rightarrow \Hom_{R}(M_{i},(\cone{\pi})_{a-1})\rightarrow \Hom_{R}(M_{i},\text{im}{\partial_{a-1}})\rightarrow 0\]
is exact. Also,
\[0\rightarrow \text{im}{\partial_{a-1}}\rightarrow (\cone{\pi})_{a-2}\rightarrow \text{im}{\partial_{a-2}}\rightarrow 0\]
is exact and $P_{R}(M_{i},(\cone{\pi})_{a-2})=P_{R}(M_{i},\text{im}{\partial_{a-1}})=0$, and so
\[0\rightarrow \Hom_{R}(M_{i},\text{im}{\partial_{a-1}})\rightarrow \Hom_{R}(M_{i},(\cone{\pi})_{a-2})\rightarrow \Hom_{R}(M_{i},\text{im}{\partial_{a-2}})\rightarrow 0\]
is exact and $P_{R}(M_{i},\text{im}{\partial_{a-2}})=0$. Therefore
\begin{align*}
0\rightarrow \Hom_{R}(M_{i},(\cone{\pi})_{a}) &\rightarrow \Hom_{R}(M_{i},(\cone{\pi})_{a-1}) \\
&\rightarrow \Hom_{R}(M_{i},(\cone{\pi})_{a-2})\rightarrow \Hom_{R}(M_{i},\text{im}{\partial_{a-2}})\rightarrow 0
\end{align*}
is exact. Continuing in this way finishes the proof.
\end{proof}

\begin{lemma} \label{bmrbmwhenmismod} Let $R$ be a Noetherian ring and let $M$ be a finitely generated module with $P_{R}(M,R)<\infty$. Then $b_{M}^{R}<\infty$ if and only if $B_{M}^{R}<\infty$.
\end{lemma}

\begin{proof}
Since $b_{M}^{R}\leq B_{M}^{R}$ is always true, we just need to show $b_{M}^{R}<\infty$ implies $B_{M}^{R}<\infty$. Let $L=\Omega^{P_{R}(M,R)}(M)$ so that $P_{R}(L,R)=0$. Since $b_{M}^{R}<\infty$ if and only if $b_{L}^{R}<\infty$, and since $B_{M}^{R}<\infty$ if and only if $B_{L}^{R}<\infty$ by Lemma \ref{prelemmaforpmrbmrimpliesgdimforcomplexes}, it is enough to show $b_{L}^{R}<\infty$ implies $B_{L}^{R}<\infty$.\newline

Let $N\in D_{f}^{b}(R)$ be a complex with semiprojective resolution $Q\rightarrow N$ where $\inf{Q}=\inf{H(N)}$ and $Q_{i}$ is finitely generated for all $i$. Let $n\geq\sup{H(N)}$. We can assume $\inf{H(N)}=0$. Now consider the exact sequence
\[\star\quad\quad 0\rightarrow Q_{\leq n-1}\rightarrow Q\rightarrow Q_{\geq n}\rightarrow 0\]
and let $Q_{\leq n-1}\rightarrow I$ be a semiinjective resolution with $I$ bounded to the left. By Lemma \ref{lemmaforbmrbmwhenmismod} $\Hom_{R}(L,Q_{\leq n-1})$ and $\Hom_{R}(L,I)$ are quasiisomorphic, and so
\[\ext_{R}^{i}(L,Q_{\leq n-1})\cong H_{-i}(\Hom_{R}(L,I))\cong H_{-i}(\Hom_{R}(L,Q_{\leq n-1}))=0\]
for $i>0$. Since $\star$ is exact and since $Q_{\geq n}\simeq\Sigma^{n}C_{n}(Q)$, 
\[0\rightarrow \mathbb{R}\Hom_{R}(L,Q_{\leq n-1})\rightarrow \mathbb{R}\Hom_{R}(L,N)\rightarrow \mathbb{R}\Hom_{R}(L,\Sigma^{n}C_{n}(Q))\rightarrow 0\]
is exact, and so
\[\ext_{R}^{i}(L,N)\cong\ext_{R}^{i}(L,\Sigma^{n}C_{n}(Q))\cong\ext_{R}^{i+n}(L,C_{n}(Q))\]
for $i>0$, proving $b_{L}^{R}<\infty$ implies $B_{L}^{R}<\infty$.
\end{proof}

\begin{proposition} \label{lemmaforpmrbmrimpliesgdimforcomplexes} Let $R$ be a Noetherian ring, let $M\in D_{f}^{b}(R)$ be a complex with $P_{R}(M,R)<\infty$, and let $P\rightarrow M$ be a semi projective resolution with $\inf{P}=\inf{H(M)}$ and $P_{i}$ finitely generated for all $i$. If $n\geq\sup{H(M)}$, then $B_{M}^{R}<\infty$ if and only if $b_{C}^{R}<\infty$, where $C=C_{n}(P)$. In particular, $B_{M}^{R}<\infty$ if $R$ has AC. 
\end{proposition}

\begin{proof}
The lemma is an immediate consequence of Lemma \ref{prelemmaforpmrbmrimpliesgdimforcomplexes} and Lemma \ref{bmrbmwhenmismod}.
\end{proof}

\section{Gorenstein dimension and a Result of Araya} \label{aresultofaraya}
In this section we prove Theorem \ref{auslanderboundandgorensteindimensionareequal}, which shows that the Auslander bound and the Gorenstein dimension of a complex are equal when both are finite, and Theorem \ref{firstauslanderbuchsbaumformula}, which generalizes a result of Araya. The idea for the following proof came from \cite[Lemma 2.10]{christensenjorgensen2}.
\begin{lemma} \label{dimensionshiftinglemma} Let $R$ be a Noetherian ring and let $M$ and $0\not\simeq N$ be complexes in $D_{b}^{f}(R)$ with $\gdim_{R}M<\infty$ and $N$ bounded. If $L\rightarrow N$ is a semiprojective resolution, then for each $n>\sup{N}$ we have
\[\widehat{\ext}_{R}^{i}(M,N)\cong\widehat{\ext}_{R}^{i+n}(M,C_{n}(L))\]
for all $i\in\mathbb{Z}$.
\end{lemma}

\begin{proof}
Let $T\rightarrow P\rightarrow M$ be a complete projective resolution of $M$ and let $L\rightarrow N$ be a semiprojective resolution of $N$ with $\inf{L}=\inf{H(N)}$ so that $L$ is bounded below. Now choose $n>\sup{N}$ so that there is a quasiisomorphism $\widetilde{\pi}:L_{\subset n}\rightarrow N$. Since $\widetilde{\pi}$ is a quasiisomorphism, $\cone{\widetilde{\pi}}$ is exact. Since $\cone{\widetilde{\pi}}$ is also bounded, $\Hom_{R}(T,\cone\widetilde{\pi})$ is exact by \cite[Lemma 2.4]{christensenfrankildholm}, and so the exact sequence
\[0\rightarrow N\rightarrow \cone\widetilde{\pi}\rightarrow \Sigma L_{\subset n}\rightarrow 0\]
implies
\[\widehat{\ext}_{R}^{i}(M,N)\cong \widehat{\ext}_{R}^{i}(M, L_{\subset n})\]
for all $i\in\mathbb{Z}$ by \cite[Proposition 4.6]{veliche}. Now consider the exact sequence
\[\star\quad\quad\quad 0\rightarrow L_{\leq n-1}\rightarrow L_{\subset n}\rightarrow \Sigma^{n}C_{n}(L)\rightarrow 0.\]
Since $L_{\leq n-1}$ is bounded, we have $\widehat{\ext}_{R}^{i}(\_,L_{\leq n-1})=0$ for all $i\in\mathbb{Z}$ by \cite[Theorem 4.5]{veliche}. Therefore $\star$ and \cite[Proposition 4.6]{veliche} gives the second isomorphism below
\[\widehat{\ext}_{R}^{i}(M,N) \cong \widehat{\ext}_{R}^{i}(M,L_{\subset n}) \cong \widehat{\ext}_{R}^{i}(M,\Sigma^{n}C_{n}(L)) \cong \widehat{\ext}_{R}^{i+n}(M,C_{n}(L)).\qedhere\]
\end{proof}

\begin{theorem} \label{auslanderboundandgorensteindimensionareequal} Let $R$ be a Noetherian ring and let $M\in D_{b}^{f}(R)$ be a complex with $B_{M}^{R}<\infty$ and $\gdim_{R}M<\infty$. Then
\[B_{M}^{R} = \gdim_{R}M.\]
\end{theorem}

\begin{proof}
Since 
\[\gdim_{R}M = P_{R}(M,R) \leq B_{M}^{R}\]
is always true, we just need to show $B_{M}^{R} \leq \gdim_{R}M$. Let $N\in D_{b}^{f}(R)$ be a complex with $P_{R}(M,N)<\infty$. We can assume $N$ is bounded and $\inf{N}=0$. Suppose $P_{R}(M,N)>\gdim_{R}M$. If $n>\sup{N}$, Lemma \ref{dimensionshiftinglemma} gives the second isomorphism below
\[\ext_{R}^{i}(M,N) \cong \widehat{\ext}_{R}^{i}(M,N) \cong \widehat{\ext}_{R}^{i+n}(M,C_{n}(L)) \cong \ext_{R}^{i+n}(M,C_{n}(L))\]
The first and third isomorphisms are for $i>\gdim_{R}M$ by \cite[4.1.1]{christensenjorgensen2}, and so $P_{R}(M,C_{n}(L))=P_{R}(M,N)+n$ for $n>\sup{N}$, contradicting the fact that $B_{M}^{R}<\infty$.
\end{proof}

\begin{corollary} \label{auslanderboundandgorensteindimensionareequalcorollary} Let $R$ be a Noetherian ring and let $M$ be a finitely generated module with $B_{M}^{R}<\infty$ and $\gdim_{R}M<\infty$. Then
\[B_{M}^{R} = \gdim_{R}M = b_{M}^{R}\]
\end{corollary}

\begin{proof}
Just note
\[b_{M}^{R} \leq B_{M}^{R} = \gdim_{R}M = P_{R}(M,R) \leq b_{M}^{R},\]
where the first equality is by Theorem \ref{auslanderboundandgorensteindimensionareequal}.
\end{proof}

Christensen and Jorgensen proved the next result for modules over AB rings in \cite[Proposition 3.2]{christensenjorgensen}.
\begin{lemma} \label{prmnfiniteimpliesvanishingtatecoho2} Let $R$ be a Noetherian ring and let $M$ and $0\not\simeq N$ be complexes in $D_{b}^{f}(R)$ with $B_{M}^{R}<\infty$, $\gdim_{R}M<\infty$, and $N$ bounded. If $P_{R}(M,N)<\infty$, then $\widehat{\ext}_{R}^{i}(M,N)=0$ for all $i\in\mathbb{Z}$.
\end{lemma}

\begin{proof}
We can assume $\inf{N}=0$. Let $T\rightarrow P\rightarrow M$ be a complete projective resolution. Then $\gdim_{R}C_{n}(T)=0$, and $B_{C_{n}(T)}^{R}<\infty$ by Lemma \ref{prelemmaforpmrbmrimpliesgdimforcomplexes}. Now fix $i\in\mathbb{Z}$ and let $n=i-1$. Then 
\[\widehat{\ext}_{R}^{i}(M,N)\cong\widehat{\ext}_{R}^{i-n}(C_{n}(T),N)\cong\widehat{\ext}_{R}^{1}(C_{n}(T),N)\cong\ext_{R}^{1}(C_{n}(T),N)=0.\]
The first isomorphism is by \cite[Lemma 4.3]{christensenjorgensen2}, for the third isomorphism see \cite[4.1.1]{christensenjorgensen2}, and the equality is by Theorem \ref{auslanderboundandgorensteindimensionareequal}. This finishes the proof.
\end{proof}

In \cite[Theorem 1.2.2.1 and Lemma 2.5]{araya} Araya proved the module version of the next result, it was first proved for modules over AB rings in \cite[Theorem 3.6]{christensenjorgensen}.
\begin{theorem} \label{firstauslanderbuchsbaumformula} Let $R$ be a Noetherian local ring and let $M$ and $0\not\simeq N$ be complexes in $D_{b}^{f}(R)$ with $B_{M}^{R}<\infty$ and $\gdim_{R}M<\infty$. If $P_{R}(M,N)<\infty$, then
\[P_{R}(M,N)=\depth{R}-\depth{M}.\]
\end{theorem}

\begin{proof}
We can assume $N$ is bounded. The first equality below is by \cite[Theorem 14.3.28]{christensenfoxbyholm}
\[\width{N}-\width{\mathbb{R}\Hom_{R}(M,N)} = \inf H({N})-\inf H({\mathbb{R}\Hom_{R}(M,N)}) = P_{R}(M,N).\]
Also, Lemma \ref{prmnfiniteimpliesvanishingtatecoho2} gives $\widehat{\ext}_{R}^{i}(M,N)=0$ for all $i\in\mathbb{Z}$, and so \cite[Proposition 6.5]{christensenjorgensen} gives the second equality below
\[P_{R}(M,N)=\width{N}-\width{\mathbb{R}\Hom_{R}(M,N)}=\depth{R}-\depth{M}.\qedhere\]
\end{proof}

The next example shows that $b_{M}^{R}$ and $B_{M}^{R}$ are different in general when $M$ is a module.
\begin{example} \label{bigandsmallauslanderboundexample} Let $R$ be a Noetherian local ring that is not CM and let $k$ be the residue field of $R$. Since $R$ is CM if and only if $R$ has a nonzero finitely generated module of finite injective dimension, we have $b_{k}^{R}=-\infty$ by definition. To compute $B_{k}^{R}$, note that if $M\in D_{b}^{f}(R)$ is a complex with $P_{R}(k,M)<\infty$, then $\idim_{R}M<\infty$ by \cite[Theorem 16.4.8]{christensenfoxbyholm}, and so $P_{R}(k,M)=\depth{R}$ by \cite[Corollary 16.4.11]{christensenfoxbyholm}. Now if $\underline{x}$ is a minimal generating set for the maximal ideal of $R$, if $H$ is the Koszul complex on $\underline{x}$, and if $E$ is the injective hull of $k$, then $\Hom_{R}(H,E)\in D_{b}^{f}(R)$ and $\idim_{R}\Hom_{R}(H,E)<\infty$, and so $B_{k}^{R}=\depth{R}$. Note that $b_{k}^{R}$ and $B_{k}^{R}$ are both finite.
\end{example}

Lemma \ref{bmrbmwhenmismod}, Corollary \ref{auslanderboundandgorensteindimensionareequalcorollary}, and Example \ref{bigandsmallauslanderboundexample} raise the following question.
\begin{question}
Let $R$ be a Noetherian ring and let $M$ be a finitely generated module. Do we always have that $b_{M}^{R}$ and $B_{M}^{R}$ are simultaneously finite? Under what conditions do we have $b_{M}^{R}=B_{M}^{R}$?
\end{question}

\section{Projective and Gorenstein Dimension} \label{thetheoremsofchristensenandholmrevisited}
In this section we prove Theorem \ref{projectivedimension} and Theorem \ref{gorensteindimension}, which extend two results of Christensen and Holm to complexes with finite Auslander bound. The results of this section were proved for AC rings in \cite{christensenholm2}.
\begin{lemma} \label{firstlemmaoflarsandhenrik} Let $R$ be a Noetherian ring, let $U$ be an exact complex, let $C$ be a finitely generated module with $b_{C}^{R}<\infty$, and assume the following
\begin{enumerate}
\item $U_{v}$ is finitely generated for all $v\gg 0$
\item $\ext_{R}^{n}(C,U_{v})=0$ for all $n\geq 1$ and all $v\in\mathbb{Z}$
\item there exists $w\in\mathbb{Z}$ so that $\ext_{R}^{n}(C,Z_{w}(U))=0$ for all $n\gg 0$.
\end{enumerate}
Then $\ext_{R}^{n}(C,Z_{v}(U))=0$ for all $n\geq 1$ and all $v\in\mathbb{Z}$. In particular, $\Hom_{R}(C,U)$ is exact.
\end{lemma}

\begin{proof}
The proof is exactly the same as the proof of \cite[Lemma 2.1]{christensenholm2}, as \cite[Lemma 2.1]{christensenholm2} only needs $C$ to have a finite Auslander bound.
\end{proof}

\begin{theorem} \label{projectivedimension} Let $R$ be a Noetherian ring and let $M\in D_{b}^{f}(R)$ be a complex. If $B_{M}^{R}<\infty$, $P_{R}(M,R)<\infty$, and $P_{R}(M,M)<\infty$, then $\pdim_{R}M<\infty$ and
\[\pdim_{R}M = -\inf{H(\mathbb{R}\Hom_{R}(M,R))}.\]
\end{theorem}

\begin{proof}
Let $L\rightarrow M$ be a semiprojective resolution and let
\[s=\max\{\sup{H(M)}, -\inf{H(\mathbb{R}\Hom_{R}(M,R))}\}.\]
By Proposition \ref{lemmaforpmrbmrimpliesgdimforcomplexes} $B_{M}^{R}<\infty$ implies $b_{C_{s}(L)}^{R}<\infty$. The proof is now exactly the same as the proof of \cite[Theorem 2.3]{christensenholm2}, just use Lemma \ref{firstlemmaoflarsandhenrik} instead of \cite[Lemma 2.1]{christensenholm2}.
\end{proof}

\begin{lemma} \label{secondlemmaoflarsandhenrik} Let $R$ be a Noetherian ring, let $T$ be a finitely generated module, let $M$ and $N$ be complexes in $D_{b}^{f}(R)$ with $P_{R}(M,R)<\infty$, and consider the tensor evaluation morphism
\[\omega_{MTN}:\mathbb{R}\Hom_{R}(M,T)\otimes_{R}^{\mathbb{L}}N\rightarrow \mathbb{R}\Hom_{R}(M,T\otimes_{R}^{\mathbb{L}}N).\]
If $B_{M}^{R}<\infty$ and the three complexes
\[\mathbb{R}\Hom_{R}(M,T),\text{ }T\otimes_{R}^{\mathbb{L}}N,\text{ and }\mathbb{R}\Hom_{R}(M,T\otimes_{R}^{\mathbb{L}}N)\]
are homologically bounded, then $\omega_{MTN}$ is an isomorphism.
\end{lemma}

\begin{proof}
Let $P\rightarrow M$ be a semiprojective resolution and let 
\[s=\max\{\sup{H(M)}, -\inf{H(\mathbb{R}\Hom_{R}(M,T))}\}.\]
By Proposition \ref{lemmaforpmrbmrimpliesgdimforcomplexes} $B_{M}^{R}<\infty$ implies $b_{C_{s}(P)}^{R}<\infty$. The proof is now exactly the same as \cite[Lemma 4.1]{christensenholm2}, as \cite[Lemma 4.1]{christensenholm2} only needs $C_{s}(P)$ to have a finite Auslander bound.
\end{proof}

\begin{theorem} \label{gorensteindimension} Let $R$ be a Noetherian local ring and let $M\in D_{b}^{f}(R)$ be a complex. If $B_{M}^{R}<\infty$, then
\[\gdim_{R}M = -\inf{H(\mathbb{R}\Hom_{R}(M,R))}.\]
\end{theorem}

\begin{proof}
The proof is exactly the same as the proof of \cite[Theorem 4.4]{christensenholm2}, just use Lemma \ref{secondlemmaoflarsandhenrik} instead of \cite[Lemma 4.1]{christensenholm2}
\end{proof}

\begin{remark} As stated \cite[Theorem 4.4]{christensenholm2} is incorrect. The proof of this result says that the AC condition localizes, which is false, see \cite[Theorem A]{nassehwagstafftakahashivandebogert}.
\end{remark}

\section{Derived Depth Formula} \label{deriveddepthformulasection}
Our definition of the following is motivated by \cite[Theorem 8.3.11]{christensenfoxbyholm}.
\begin{definition} Let $M$ and $N$ be a complexes and define
\[q_{R}(M,N)=\sup\{n\in\mathbb{Z}|\tor_{n+\sup{H(N)}}^{R}(M,N)\neq 0\}.\]
\end{definition}

Note that $q_{R}(M,N)$ is invariant under shifting $N$. Christensen and Jorgensen proved the next result for modules over AB rings in \cite[Proposition 3.2]{christensenjorgensen}.
\begin{lemma} \label{lemmatwoforgdimcomplexes} Let $R$ a Noetherian local ring, let $M$ be a finitely generated module with $\gdim_{R}M=0$ and $B_{M}^{R}<\infty$, and let $N\in D_{b}^{f}(R)$ be a bounded complex. If $q_{R}(M,N)<\infty$, then $\widehat{\tor}_{n}^{R}(M,N)=0$ for $n\in\mathbb{Z}$.
\end{lemma}

\begin{proof}
We can assume $\sup{N}=0$. Let 
\[\dots\rightarrow T_{1}\xrightarrow{\partial_{1}} T_{0}\xrightarrow{\partial_{0}} T_{-1}\rightarrow\dots\]
be an exact complex of finitely generated free modules with $M=\coker{\partial_{1}}$ and let $F\rightarrow N$ be a semiprojective resolution with $\inf{F}=\inf{H(N)}$ and $F_{i}$ finitely generated for all $i$. Fix $n\in\mathbb{Z}$. Then $C_{n}=\coker{\partial_{n}}$ is so that $\gdim_{R}C_{n}=0$, $q_{R}(C_{n},N)<\infty$, and
\[\widehat{\tor}_{n}^{R}(M,N)\cong\widehat{\tor}_{1}^{R}(C_{n},N)\cong\tor_{1}^{R}(C_{n},N)\]
(the second isomorphism is by \cite[2.4.1]{christensenjorgensen2}). We also have $B_{C_{n}}^{R}<\infty$ by Lemma \ref{prelemmaforpmrbmrimpliesgdimforcomplexes}. Since
\[0\rightarrow F_{\leq -1}\rightarrow F\rightarrow F_{\geq 0}\rightarrow 0\]
is exact and since $F_{\geq 0}\simeq C_{0}(F)$, we have
\[\tor_{r}^{R}(C_{n},N)\cong \tor_{r}^{R}(C_{n},C_{0}(F))\]
for all $r>0$, and so $q_{R}(C_{n},C_{0}(F))<\infty$. Therefore $q_{R}(C_{n},C_{0}(F))=0$ by Lemma \ref{bmrbmwhenmismod} and \cite[Proposition 7.5.5]{sanders}, which implies $\widehat{\tor}_{n}^{R}(M,N)=0$.
\end{proof}

When the ring is AB, Christensen and Jorgensen proved the module version of the next result in \cite[Theorem 2.3 and Remark 3.3]{christensenjorgensen}.
\begin{theorem} \label{deriveddepthformulaforcomplexes} Let $R$ be a Noetherian local ring and let $M$ and $N$ be complexes in $D_{b}^{f}(R)$ with $\gdim_{R}M<\infty$ and $B_{M}^{R}<\infty$. If $q_{R}(M,N)<\infty$, then
\[\depth{M\otimes_{R}^{\mathbb{L}}N}=\depth{M}+\depth{N}-\depth{R}.\]
\end{theorem}

\begin{proof}
Let $P\rightarrow M$ be a semiprojective resolution with $\inf{P}=\inf{H(M)}$ and $P_{i}$ finitely generated for all $i$. Let $L=C_{i}(P)$ where $i>\max\{\sup{H(M)},P_{R}(M,R)\}$. Then $\gdim_{R}L=0$. Also, we can assume $N$ is bounded. Now \cite[Lemma 2.10]{christensenjorgensen2} gives the first equality below and Lemma \ref{lemmatwoforgdimcomplexes} gives the second equality below for all $n\in\mathbb{Z}$
\[\widehat{\tor}_{n}^{R}(M,N)=\widehat{\tor}_{n-i}^{R}(L,N)=0.\]
Applying \cite[Theorem 2.3]{christensenjorgensen} finishes the proof.
\end{proof}


\end{document}